\documentclass[12pt,draftcls,onecolumn]{IEEEtran}
\usepackage{epsf}
\usepackage{bm}
\usepackage{latexsym}
\usepackage{amsmath}
\usepackage{amssymb}
\usepackage{amsthm}
\usepackage{graphicx}
\usepackage{subfigure}
\usepackage{epstopdf}
\usepackage{color}
\usepackage{mathtools}
\usepackage{setspace}
\usepackage{cite}
\RequirePackage{amsopn}
\RequirePackage{amsfonts}

\mathtoolsset{showonlyrefs}

\title{\LARGE \bf A Computationally Efficient Implementation of Fictitious Play for Large-Scale Games}
\author{BRIAN SWENSON$^{\dagger\star}$, SOUMMYA KAR$^\dagger$ AND JO\~{A}O XAVIER$^\star$\thanks{The work was partially supported by the FCT project FCT [UID/EEA/50009/2013] through the Carnegie-Mellon/Portugal Program managed by ICTI from FCT and by FCT Grant CMU-PT/SIA/0026/2009, and was partially supported by NSF grant ECCS-1306128. \newline$^\dagger$Department of Electrical and Computer Engineering, Carnegie Mellon University, Pittsburgh, PA 15213, USA (brianswe@andrew.cmu.edu and soummyak@andrew.cmu.edu).\newline $^\star$Institute for Systems and Robotics (ISR/IST), LARSyS, Instituto Superior T\'{e}cnico, Portugal (jxavier@isr.ist.utl.pt).}}

\newtheorem{theorem}{Theorem}
\newtheorem{lemma}{Lemma}

\newtheorem{assumption}{A.}

\begin{document}

\maketitle
\thispagestyle{empty}
\begin{abstract}
The paper is concerned with distributed learning and optimization in large-scale settings. The well-known Fictitious Play (FP) algorithm has been shown to achieve Nash equilibrium learning in certain classes of multi-agent games. However, FP can be computationally difficult to implement when the number of players is large. Sampled FP is a variant of FP that mitigates the computational difficulties arising in FP by using a Monte-Carlo (i.e., sampling-based) approach.
The Sampled FP algorithm has been studied both as a tool for distributed learning and as an optimization heuristic for large-scale problems.
Despite its computational advantages, a shortcoming of Sampled FP is that the number of samples that must be drawn in each round of the algorithm grows without bound (on the order of $\sqrt{t}$, where $t$ is the round of the repeated play).
In this paper we propose Computationally Efficient Sampled FP (CESFP)---a variant of Sampled FP in which only one sample need be drawn each round of the algorithm (a substantial reduction from $O(\sqrt{t})$ samples per round, as required in Sampled FP). CESFP operates using a stochastic-approximation type rule to estimate the expected utility from round to round. It is proven that the CESFP algorithm achieves Nash equilibrium learning in the same sense as classical FP and Sampled FP. Simulation results suggest that the convergence rate of CESFP (in terms of repeated-play iterations) is similar to that of Sampled FP.
\end{abstract}
\section{Introduction}

A game-theoretic learning algorithm is an adaptive multi-agent procedure which can enable a system of interacting agents to achieve desirable global behavior using local (agent-based) control laws. Such algorithms have a wide range of applications in distributed control \cite{marden2012game,saad2012game,mrp-2011-journal,wang2010game,arslan2007autonomous} and large-scale optimization \cite{yang2010distributed,na-marden-2011b,marden-young-2011,Lambert01,garcia2000fictitious}.

Fictitious Play (FP) \cite{Brown51} is an archetypal game-theoretic learning algorithm that has received much attention over the years due to its intuitively simple nature and proven convergence results for certain important classes of games (see, for example, \cite{fudenberg1998theory,young2004strategic} and references therein).
Of particular interest are results demonstrating that FP leads players\footnote{We use the terms agent and player interchangeably throughout the paper.} to learn equilibrium strategies in potential games---a class of multi-agent games in which there may be an arbitrarily large number of players \cite{Mond01,Mond96}.

However, such convergence results tend to be of limited practical value due to computational difficulties that may arise when implementing FP in large games. In particular, in each stage of the FP algorithm, each player $i$ must compute the expected (mixed) utility for each of her actions given her current beliefs regarding opponents' strategies. Evaluating this expected utility---the domain of which is an $(n-1)$-dimensional probability simplex---is a problem whose complexity in general scales exponentially in terms of the number of players, $n$.

The main focus of the present paper is the presentation of a variant of FP that might be more practical to implement in certain large-scale settings. In particular, we consider a practical method for mitigating computational complexity using a Monte-Carlo type approach.

Sampled FP \cite{Lambert01,epelman2011sampled,ghate2014sampled,lambert2003fictitious,garcia2000fictitious,sisikoglu2011sampled} introduced the idea of mitigating complexity in FP using a Monte-Carlo (i.e., sampling-based) approach. At each iteration of the Sampled FP algorithm, players approximate the expected utility by drawing several samples from an underlying probability distribution. Players then myopically choose an ``optimal'' next-stage action using the approximated utility as a surrogate for the true expected utility. The work \cite{Lambert01} showed that, as long as the number of samples drawn each round grows sufficiently quickly, players learn an equilibrium in the same sense as FP, almost surely (a.s.).

In essence, Sampled FP achieves a mitigation in complexity by avoiding any direct evaluation of the expected utility. However, Sampled FP has a notable shortcoming: In order to guarantee learning is achieved, the number of samples that must be drawn in each iteration (i.e., round) of the algorithm grows without bound (on the order of $\sqrt{t}$ samples per round, where $t$ is the current round of the repeated play algorithm).

In this paper, we propose a variant of Sampled FP---which we call Computationally Efficient Sampled FP (CESFP)---in which only one sample need be drawn each round of the repeated play process. CESFP achieves the same fundamental computational advantage as Sampled FP (i.e., direct evaluation of the expected utility is avoided) but does so by drawing only one sample per round (rather than the $O(\sqrt t)$ samples per round, required in Sampled FP).

Intuitively, the reduction in the required number of per-round samples is accomplished by treating the expected utility process as \emph{quasi static}. Such treatment is possible due to the diminishing incremental step size in the the expected utility process.
In CESFP, the sample data gathered in the current round of repeated play is recursively combined with sample data from previous rounds using a stochastic-approximation-type estimation rule. This may be contrasted with Sampled FP where, in each round of the repeated play, data gathered from sampling in the previous round is wholly discarded and a fresh set of samples is gathered to  approximate the expected utility for the upcoming round. (See Section \ref{sec_CESFP_discussion} for more details.)

Due to the improved efficiency in information handling, CESFP is able to achieve convergence at a rate similar to that of sampled FP (in terms of repeated-play iterations) despite drawing far fewer samples per-iteration. (See Section \ref{sec_simulations} for more details.)

The main contribution of the paper is the presentation of the CESFP algorithm and proof of convergence of the algorithm in terms of empirical frequency to the set of Nash equilibria (a.s.). The proof relies on showing that the CESFP process may be seen as a Generalized Weakened FP process as studied in \cite{leslie2006generalised}.

CESFP may be applicable as a computationally efficient variant of FP in a variety of settings including large-scale optimization \cite{Lambert01},\cite{lambert2003fictitious}, dynamic programming \cite{epelman2011sampled,ghate2014sampled}, traffic routing \cite{garcia2000fictitious}, and cognitive radio \cite{dabcevic2014fictitious,dabcevic2014intelligent,wang2010game}, and learning in Markov decision processes \cite{sisikoglu2011sampled}. CESFP may also be used as a general tool for distributed learning \cite{swenson2015D-CESFP} or control \cite{marden2012game}.

Related works have studied approaches for mitigating computational issues arising in large-scale implementations of FP. Joint Strategy FP (JSFP) \cite{marden06} studies a variant of FP in which players update a utility estimate using a computationally-simple recursive procedure and choose next-stage actions using a best-response rule combined with an inertial term. JSFP is shown to converge to pure strategy Nash equilibria (NE) in ordinal potential games but is fundamentally different from CESFP (and FP) in that the tracked utility corresponds to an empirical distribution taken over joint actions,\footnote{In FP, Sampled FP, and CESFP, players best respond to the product of marginal empirical distributions (or an estimate thereof) which implicitly presumes a form of independence among opponents' strategies. Tracking and responding to the empirical distribution of joint actions, as in JSFP, fundamentally alters the dynamics of classical FP. CESFP achieves computationally efficiency while preserving the basic dynamical structure of FP.} and convergence may only occur at pure NE.

Payoff-based learning algorithms---including those based on FP \cite{chapman2013convergent,leslie2006generalised,leslie2003convergent,Arslan04} and otherwise \cite{marden-payoff,germano2007global,pradelski2012learning}---tend to be computationally simple in large games and have the further advantage that they do not require players to have any knowledge of the game's utility structure.
However, such algorithms implicitly assume players have access to instantaneous payoff information and may not be applicable in settings where this information is costly to obtain, delayed, or otherwise unavailable.

For example, in a follow-up paper \cite{swenson2015D-CESFP} we consider an application in which CESFP is implemented in a network-based setting in which all inter-agent communication is restricted to a preassigned (possibly sparse) communication graph~\cite{swenson2012ECFP}. Instantaneous payoff information can be difficult to obtain in such a setting, particularly in the case that the utility corresponds to a non-local welfare-type utility function, which may not be physically measurable at any single agent. Furthermore, there are circumstances in which each round of physical game play can incur an exogenous cost. In such cases it may be preferable to supplement payoff-based learning---which depends on interaction in the physical environment---with forms of model-based learning.

The remainder of the paper is organized as follows. Section \ref{sec_prelims} sets up the notation to be used in the subsequent development. Section \ref{sec_FP} reviews classical FP. Section \ref{sec_sampled_FP} reviews Sampled FP. Section \ref{sec_CESFP} presents the CESFP algorithm, states the main convergence result for CESFP, and proves the result. Section \ref{sec_simulations} presents a simulation example comparing Sampled FP and CESFP. Section \ref{sec_conclusions} provides concluding remarks.

\section{Preliminaries}
\label{sec_prelims}
A game in normal form is represented by the tuple $\Gamma := (N,(Y_i,u_i)_{i\in N})$, where $N = \{1,\ldots,n\}$ denotes the set of players, $Y_i$ denotes the finite set of actions available to player $i$, and $u_i:\prod_{i\in N}Y_i \rightarrow \mathbb{R}$ denotes the utility function of player $i$. Denote by $Y:= \prod_{i\in N} Y_i$ the joint action space.

In order to guarantee the existence of Nash equilibria it is necessary to consider the mixed-extension of $\Gamma$ in which players are permitted to play probabilistic strategies. Let $m_i := |Y_i|$ be the cardinality of the action space of player $i$, and let $\Delta_i := \{p\in \mathbb{R}^{m_i}:\sum_{k=1}^{m_i}p(k) = 1,~p(k)\geq 0, ~\forall k\}$ denote the set of mixed strategies available to player $i$---note that a mixed strategy is a probability distribution over the action space of player $i$. Denote by $\Delta^n := \prod_{i\in N} \Delta_i$ the set of joint mixed strategies. When convenient, we represent a mixed strategy $p\in \Delta^n$ by $p=(p_i,p_{-i})$, where $p_i$ denotes the marginal strategy of player $i$ and $p_{-i}$ is a $(n-1)$-tuple containing the marginal strategies of the other players.

In the context of mixed strategies, we often wish to retain the notion of playing a single deterministic action. For this purpose, let $A_i := \{e_1,\ldots,e_{m_i}\}$ denote the set of ``pure strategies'' of player $i$, where $e_j$ is the $j$-th canonical vector containing a $1$ at position $j$ and zeros otherwise. Note that there is a one-to-one correspondence between a player's action set $Y_i$ and the player's set of pure strategies $A_i\subset \Delta_i$.

%
%

The mixed utility function of player $i$ is given by
\begin{equation}
\label{def_mixed_U}
U_i(p) := \sum_{y \in Y} u_i(y) p_1(y)\ldots p_n(y)
\end{equation}
\noindent where $U_i:\Delta^n \rightarrow \mathbb{R}$. Note that the mixed utility $U_i(p)$ may be interpreted as the expected utility of $u_i(y)$ given that players' (marginal) mixed strategies $p_i$ are independent.

The set of Nash equilibria is given by $NE := \{p\in \Delta^n: U_i(p_i,p_{-i}) \geq U_i( p_i',p_{-i}), ~\forall p_i' \in \Delta_i,~\forall i\in N\}$.
The distance of a distribution $p \in \Delta^n$ from a set $S \subset \Delta^n$ is given by $d(p,S) = \inf \{ \| p - p' \| : p'\in S\}$. Throughout the paper $\|\cdot\|$ denotes the standard $\mathcal{L}_{2}$ Euclidean norm unless otherwise specified.

Throughout, we assume there exists a probability space $(\Omega,\mathcal{F},\mathbb{P})$ rich enough to carry out the construction of the various random variables required in the paper. For a random object $X$ defined on a measurable space $(\Omega,\mathcal{F})$, let $\sigma(X)$ denote the $\sigma$-algebra generated by $X$ \cite{williams_book}. As a matter of convention, all equalities and inequalities involving random objects are to be interpreted almost surely (a.s.) with respect to the underlying probability measure, unless otherwise stated.

\subsection{Repeated Play}
The learning algorithms considered in this paper all assume the following format of repeated play. Let a normal form game $\Gamma$ be fixed. Let players repeatedly face off in the game $\Gamma$, and for $t\in\{1,2,\ldots\}$, let $a_i(t)\in A_i$ denote the action played by player $i$ in round $t$. Let the $n$-tuple $a(t) = (a_1(t),\ldots,a_n(t))$ denote the joint action at time $t$.

Let the empirical history distribution (or empirical distribution) of player $i$ be given by\footnote{Note that each $a_i(t)\in A_i$ is a delta distribution, and thus the empirical distribution $q_i(t)$ is a normalized histogram of the action choices of player $i$.} $q_i(t) := \frac{1}{t}\sum_{s=1}^t a_i(s)$, and let the joint empirical distribution be given by the $n$-tuple $q(t) = (q_1(t),\ldots,q_n(t))$.

\section{Fictitious Play}
\label{sec_FP}
A FP process may be intuitively described as follows: A finite set of $n$ agents engage in repeated play of some fixed normal form game. Each round of the repeated play, each agent $i$ plays an action that is myopically optimal under the (naive) assumption that all opponents are playing according to time-invariant and statistically independent strategies. In particular, under this assumption, each player $i$ believes that the empirical distribution $q_{-i}(t)$ of opponents' play is an accurate representation of opponents' (supposedly time-invariant) strategies and chooses a next-stage action that optimizes their utility given this belief.

A formal description of the FP algorithm is given below.
\subsection{FP algorithm}
\noindent \emph{Initialize} \\
(i) Each player $i$ chooses an arbitrary initial action $a_i(1)\in A_i$. The empirical distribution is initialized as $q_i(1) = a_i(1),~\forall i$. \\
\noindent \emph{Iterate} ($t\geq 1$)\\
(ii) Each player $i$ chooses her next-stage action as a best response to the current empirical distribution of opponents' play:
\begin{equation}
\label{FP_action_eq}
a_i(t+1) \in \arg\max_{\alpha_i \in A_i} U_i(\alpha_i,q_{-i}(t)).
\end{equation}
\noindent (iii) For each player $i$, the empirical distribution is updated to reflect the action just taken, $q_i(t+1) = \frac{1}{t+1}\sum_{s=1}^{t+1} a_i(s),$
or equivalently in recursive form:
$$q_i(t+1) = q_i(t) + \frac{1}{t+1}(a_i(t+1) - q_i(t)).$$
\subsection{Discussion}

Various works have established results characterizing classes of games in which FP does (or does not) lead players to learn NE strategies (see monographs \cite{fudenberg1998theory},\cite{young2004strategic}).
Of particular relevance to the large-scale setting is a class of multi-agent games known as potential games \cite{Mond96}. A game $\Gamma = (N,(Y_i,u_i(\cdot))_{i\in N})$ is said to be a potential game if there exists a potential function $\phi:Y^n\rightarrow \mathbb{R}$ such that for all $i\in N$, and all $y_{-i} \in Y_{-i}$
$$u_i(y_i,y_{-i}) - u_i(x_i,y_{-i}) = \phi(y_i,y_{-i}) - \phi(x_i,y_{-i}), ~\forall y_i,x_i \in Y_i.$$
\noindent Intuitively, the existence of a potential function means that all player's utility functions are aligned in such way that players share a common underlying objective. It has been shown \cite{Mond96,Mond01} that if $\Gamma$ is a potential game, then FP leads players to learn NE strategies in the sense that $\lim_{t\rightarrow\infty} d(q(t),NE) = 0$.

\subsection{Computational Complexity in large-scale FP}
While FP is theoretically proven to achieve NE learning in potential games, it can be computationally difficult to implement when the number of players is large.
In particular, note that in order to choose a next-stage action (see \eqref{FP_action_eq}) player $i$ must compute the mixed utility $U_i(\alpha_i,q_{-i}(t)), ~\forall \alpha_i \in A_i$. Recalling the definition of mixed utility \eqref{def_mixed_U}, this is equivalent to computing an expected value over an $(n-1)$-dimensional probability simplex. In general, the complexity of this computation grows exponentially in terms of the number of players.

\section{Sampled FP}
\label{sec_sampled_FP}
In order to mitigate the problem of computational complexity in FP, \cite{Lambert01} proposed Sampled FP. In Sampled FP, players use a Monte-Carlo approach to avoid direct evaluation of the mixed utility when choosing a next-stage action.

In particular, for each $\alpha_i \in A_i$, let $\widehat U_i(\alpha_i,t)$ denote an estimate that player $i$ forms of the mixed utility $U_i(\alpha_i,q_{-i}(t))$. Each round of play, for each player $i$, several ``test actions'' are drawn as random samples from opponents' joint empirical distribution
$q_{-i}(t)$. For each action $\alpha_i\in A_i$, player $i$ computes the average utility the action $\alpha_i$ would generate given the randomly sampled ``test actions.'' Player $i$ then chooses a next-stage action that is myopically optimal using the estimated utility as a surrogate for the true mixed utility in \eqref{FP_action_eq}. So long as the number of samples increases sufficiently quickly, it can be shown that Sampled FP leads players to learn NE strategies.

Formally, let $k_t$ denote the number of samples drawn in round $t$. The following assumption on $k_t$ is sufficient to ensure learning is achieved.
\begin{assumption}
\label{a_sampled_FP}
The number of samples drawn in round $t$ satisfies $k_t = \lceil Ct^\gamma\rceil$ where $\gamma > 1/2$ and $C>0$.
\end{assumption}
\noindent The Sampled FP algorithm is outlined below.
\subsection{Sampled FP Algorithm}
\noindent \emph{Initialize} \\
(i) Each player $i$ chooses an arbitrary initial action $a_i(1)\in A_i$. The empirical distribution is initialized as $q_i(1) = a_i(1), ~\forall i$.
\noindent\\
\emph{Iterate} ($t\geq 1$)\\
(ii) $k_t$ ``test actions'' are drawn as random samples from the joint empirical distribution $q(t)$; let $\tilde a^s(t)$ denote the $s$-th random sample drawn in round $t$. Player $i$ estimates the utility of each of her actions $\alpha_i \in A_i$ by\footnote{Since $\tilde a_{-i}(t)$ is a pure strategy, the evaluation of the utility is relatively simple. Note also that it is not necessary to draw a separate sequence of ``test actions'' $\{\tilde a^s(t)\}_{s=1}^{k_t}$ for each player $i$, although, doing so does not affect the convergence result.}
\begin{equation}
\label{sampled_FP_utility_estimate}
\widehat U_i(\alpha_i,t) = \frac{1}{k_t}\sum_{s=1}^{k_t} U_i(\alpha_i,\tilde a_{-i}^s(t)).
\end{equation}
\noindent (iii) Each player $i$ chooses a next-stage action that is a best response given her estimate of the mixed utility:
\begin{equation}
\label{sampled_FP_action_eq}
a_i(t+1) \in \arg\max_{\alpha_i \in A_i} \widehat U_i(\alpha_i,t).
\end{equation}
\noindent (iv) The empirical distribution for each player $i$ is updated recursively to account for the action just taken:
$$q_i(t+1) = q(t) + \frac{1}{t+1}\left(a_i(t+1) - q_i(t) \right).$$
\subsection{Convergence Result}
The following result (\cite{Lambert01}, Theorem 5) shows that under \textbf{A.\ref{a_sampled_FP}}, Sampled FP achieves NE learning in the same sense as classical FP for a special class of potential games known as identical interests games \cite{Mond01}, in which all players use an identical utility function; i.e., $u_i(y) = u_j(y), ~\forall y\in Y,~\forall i,j\in N.$
\begin{theorem}[Theorem 5 in \cite{Lambert01}]
Let $\Gamma$ be a finite game in strategic form with identical payoffs. Then, any Sampled FP process with sample sizes satisfying \textbf{A.\ref{a_sampled_FP}} converges in beliefs to equilibrium with probability 1. That is, $d(q(t),NE) \rightarrow 0$ almost surely as $t\rightarrow\infty$.
\end{theorem}

\section{Computationally Efficient Sampled FP}
\label{sec_CESFP}
While Sampled FP does obtain computational savings when compared to classical FP, it may be considered unsatisfactory in the sense that it requires players to draw a number of samples each round that grows without bound (see \textbf{A.\ref{a_sampled_FP}}).
In this section, we present an adaptation of Sampled FP in which only one sample need be drawn each round of the repeated play.

\subsection{Algorithm Setup}
In CESFP, players form an estimate of the mixed utility using a recursive stochastic-approximation-type rule. Similar to Sampled FP, let $\widehat U_i(\alpha_i,t)$ be the estimate which player $i$ maintains of the mixed utility $U_i(\alpha_i,q_{-i}(t))$ for each of her actions $\alpha_i \in A_i$. Let $\{\rho(t)\}_{t\geq 1}$ be a deterministic sequence of weights to be used in the stochastic-approximation-type procedure, and assume:
\begin{assumption}
The sequence $\{\rho(t)\}_{t\geq}$ is such that $0< \rho(t)\leq 1$, $\sum_{t\geq 1} (\rho(t))^2 < \infty$, and $\lim_{t\rightarrow\infty} \frac{1}{t\rho(t)} = 0$.
\label{a_rho_1}
\end{assumption}
Note that, by Lemma \ref{lemma_a_rho_implication}, \textbf{A.\ref{a_rho_1}} implies that $\sum_{t\geq 1} \rho(t) = \infty$.
The Computationally Efficient Sampled FP algorithm is outlined below.
\subsection{Computationally Efficient Sampled FP Algorithm}
\noindent \emph{Initialize}\\
(i) For each $i\in N$, let $a_i(1)\in A_i$ be arbitrary. Initialize the empirical distribution as $q_i(1) = a_i(1),~\forall i$, and initialize the utility estimate as $\widehat U(\alpha_i,0) = 0,~\forall \alpha_i \in A_i,~\forall i$. \\
\emph{Iterate} $(t\geq 1)$\\
(ii) A single ``test action'' $a^*(t)$ is drawn as a (statistically independent) random sample from the distribution $q(t)$, and each player $i$ updates the estimate $\widehat U_i(\alpha_i,t)$ for each action $\alpha_i\in A_i$ according to the recursion,\footnote{Since $a_{-i}^*(t)$ is a pure strategy, the evaluation of the utility is relatively simple. Also note that it is not necessary for each player $i$ to draw a separate ``test action'' $a_{-i}^*(t+1)$; although, if desired (for example, in a distributed setting \cite{swenson2015D-CESFP}), doing so does not affect the convergence result.}
\begin{equation}
\label{CESFP_estimate_rule}
\widehat U_i(\alpha_i,t) = (1-\rho(t))\widehat U_i(\alpha_i,t-1) + \rho(t)U_i(\alpha_i,a_{-i}^*(t)).
\end{equation}
\noindent (iii) Each player $i$ chooses a next-stage action using the rule (cf. \eqref{FP_action_eq}, \eqref{sampled_FP_action_eq}):
\begin{equation}
\label{efficient_FP_action_eq}
a_i(t+1) \in \arg\max_{\alpha_i \in A_i} \widehat U_i(\alpha_i,t).
\end{equation}
\noindent (iv) The empirical distribution for each player $i$ is updated to reflect the action just taken:
\begin{equation}
q_i(t+1) = q_i(t) + \frac{1}{t+1}(a_i(t+1) - q_i(t)).
\label{qt_CESFP}
\end{equation}
\subsection{Discussion}
\label{sec_CESFP_discussion}
The main difference between Sampled FP and CESFP is the manner in which players form estimates of the mixed utility sequence $\{U(\alpha_i,q_{-i}(t))\}_{t\geq 1},~\forall \alpha_i \in A_i$. In Sampled FP, players' estimates (see \eqref{sampled_FP_utility_estimate}) ``start afresh'' each round of the repeated play---information gathered from sampling in the previous round is discarded, and players draw roughly $\sqrt t$ new samples in order to form an estimate of the utility for the current round.

This may be considered an inefficient use of information, since the mixed utility only changes slightly from one round to the next. In particular, note that the mixed utility $U_i(\alpha_i,\cdot)$ is Lipschitz continuous with some Lipschitz constant $K$, and the increment of the empirical distribution \eqref{qt_CESFP} is bounded as $\|q(t) - q(t-1)\| \leq \frac{M}{t}$ for some constant $M>0$. Thus, the increment in the mixed utility is bounded as
\begin{equation}
|U_i(\alpha_i, q_{-i}(t)) - U_i(\alpha_i,q_{-i}(t))| \leq KM/t. 
\label{U_step_size}
\end{equation}

Intuitively speaking, this means that if one has an accurate estimate of the mixed utility $U_i(\alpha_i,q_{-i}(t-1))$ in round $(t-1)$, then it is wasteful to wholly discard this information when forming an estimate of $U_i(\alpha_i,q_{-i}(t))$.
The CESFP estimation rule leverages the diminishing increment property \eqref{U_step_size} in order to form an accurate estimate using only one sample per round.

Effectively, the Sampled FP estimation rule treats $\{U_i(\alpha_i,q_{-i}(t))\}_{t\geq 1}$ as if it were arbitrarily generated from one round to the next---drawing a completely new set of $t^\gamma$, $\gamma>1/2$ samples to estimate each $U_i(\alpha_i,q_{-i}(t))$. The CESFP estimation rule, on the other hand, treats $\{U_i(\alpha_i,q_{-i}(t))\}_{t\geq 1}$ as if it were \emph{quasi static}, drawing one sample per round, and taking a type of average over time.\footnote{Additional insight may be gained by considering the dynamical systems approach to stochastic approximations (e.g. \cite{borkar2008stochastic}) which allows one to study the behavior of certain discrete-time processes by analyzing an associated differential equation. In such an analysis, an estimation rule such as \eqref{CESFP_estimate_rule} is often considered as a two time-scale system \cite{borkar1997stochastic,leslie2003convergent}, with the ODE associated with the estimation rule operating at a \emph{faster} rate than ODE associated with the mixed utility process. In the asymptotic analysis of such systems, the slower process may often be treated as effectively static compared to the faster process. We note, however, that the proofs of our results rely primarily on self contained martingale-type arguments, rather than invoking results from dynamical systems based treatment of stochastic approximation literature.} Because of this, despite drawing only one sample per round, the CESFP estimate of $U_i(\alpha_i,q_{-i}(t))$ \emph{effectively} utilizes information from $t$ samples, while the Sampled FP estimate utilizes information from (only) $t^\gamma$, $\gamma>1/2$ samples. In practice, the CESFP and Sampled FP estimation rules tend to reduce estimation  error at comparable rates. See Section \ref{sec_simulations} for more details.

\subsection{Main Result}
The following theorem states that CESFP achieves learning in the same sense as classical FP (and Sampled FP). The result is stated for a slightly broader classes of games than previously discussed, including two-player zero-sum games \cite{fudenberg1998theory}, and generic $2\times m$ games \cite{berger2005fictitious}.
\begin{theorem}
Let $\Gamma$ be a potential game, zero-sum game, or generic $2\times m$ game. Let $\{a(t)\}_{t\geq 1}$ be a Computationally Efficient Sampled FP process, and assume \textbf{A.\ref{a_rho_1}} holds. Then $\lim_{t\rightarrow\infty} d(q(t),NE) =0$ (a.s.).
\end{theorem}
\begin{proof}
We will prove the result by showing that there exists a sequence $\{\epsilon_t\}_{t\geq 1}$ with $\lim_{t\rightarrow\infty} \epsilon_t = 0$ such that $U_i(a_i(t+1),q_{-i}(t)) \geq \max_{\alpha_i \in A_i} U_i(\alpha_i,q_{-i}(t)) - \epsilon_t.$
By \cite{leslie2006generalised}, Corollary 5, this is sufficient to guarantee $\lim_{t\rightarrow\infty} d(q(t),NE) = 0$.

Since, by \eqref{efficient_FP_action_eq}, $a_i(t+1) \in \arg\max_{\alpha_i \in A_i} \widehat U_i(\alpha_i,t)$, it is sufficient to show that, for every $i\in N$ and every $\alpha_i \in A_i$,
\begin{equation}
\label{suff_cond}
|\widehat U_i(\alpha_i,t) - U_i(\alpha_i,q_{-i}(t))| \rightarrow 0 \mbox{ as } t\rightarrow \infty ~\mbox{ a.s.}
\end{equation}
\noindent (Note that the individual action spaces $A_{i}$ are finite.) We will show this by invoking the result of Lemma \ref{lemma1} (see appendix). To that end fix $i\in N$ and $\alpha_i\in A_i$, and let $X(t) := U_i(\alpha_i,a^*_{-i}(t)),~ t\geq 1,$ $\mu(t) := U_i(\alpha_i,q_{-i}(t)),~ t\geq 1,$ and $\hat \mu(t) := \widehat U_i(\alpha_i,t),~t\geq 0$.
For $t\geq 0$, let $\mathcal{F}_t := \sigma(\{q_{-i}(s)\}_{s=1}^{t+1})$.
Note that $\mu(t)$ is $\mathcal{F}_{t-1}$-measurable and that $E(X(t)\vert\mathcal{F}_{t-1}) = \mu(t)$.

In order to invoke Lemma \ref{lemma1} it is sufficient to show that
\begin{equation}
\left(\frac{1}{\rho(t)}-1\right)\left(\mu(t) - \mu(t-1)\right) \rightarrow 0.
\label{thrm1_eq1}
\end{equation}
Let $M := \max_{q_{-i}',q_{-i}'' \in \Delta_{-i}} \|q_{-i}' - q_{-i}''\|$, and note that by \eqref{qt_CESFP},
$\|q_{-i}(t) - q_{-i}(t-1)\| \leq \frac{M}{t}.$
The utility function $U_i$ is multilinear, and hence Lipschitz continuous, so there exists a constant $K$ such that
\begin{align}
& |\mu(t) - \mu(t-1)| = |U_i(\alpha_i,q_{-i}(t)) - U_i(\alpha_i,q_{-i}(t-1))|\\
& \leq K\|q_{-i}(t) - q_{-i}(t-1)\| \leq KM/t.
\end{align}
\noindent This, together with \textbf{A.\ref{a_rho_1}}, implies that \eqref{thrm1_eq1} holds.

Thus, $X(t),~\mu(t),~\hat \mu(t),$ and $\mathcal{F}_t$ as defined above fit the template of Lemma \ref{lemma1}. By Lemma \ref{lemma1}, $
|\widehat U_i(\alpha_i,t) - U_i(\alpha_i,q_{-i}(t))| =  |\hat \mu(t) - \mu(t)| \rightarrow 0 \mbox{ as } t\rightarrow \infty,$
verifying that \eqref{suff_cond} holds.
\end{proof}

\section{Simulation Results}
\label{sec_simulations}
In order to demonstrate the computational properties of CESFP and Sampled FP in large games, we simulated both algorithms in a simple traffic routing scenario.
Let $N = \{1,\ldots n\}$ denote a finite set of drivers (or players). Drivers share a common starting point and a common destination and may travel on one of $50$ parallel routes. Let the set of routes be denoted by $R$, and let the action space of player $i$ be given by $Y_i = R, ~\forall i$. Let $\sigma_r(y)$ denote the number of drivers on route $r$ given the joint strategy $y$. Each route $r\in R$ has an associated cost function $c_r:\mathbb{N} \rightarrow \mathbb{R}$ signifying the delay experienced on route $r$ given the number of drivers using the route. Let the utility function of player $i$ be given by $u_i(y) := -c_{y_i}(\sigma_{y_i}(y))$. We note that this game is an instance of a congestion game---a known subset of potential games.

We simulated Sampled FP and CESFP in this routing scenario with 1000 drivers. In the simulation, Sampled FP used a sample rate of $k_t = \lfloor t^{.6} \rfloor$ samples per round and CESFP used the parameter $\rho(t) = t^{-.6},~\forall t$.
Figure \ref{fig_eval_time} shows the wall clock running time through iteration $t$ for each algorithm.
\begin{figure}[t]
\centering
 \vspace{-0.2cm}
    \subfigure[]{\includegraphics[width=4.35cm]{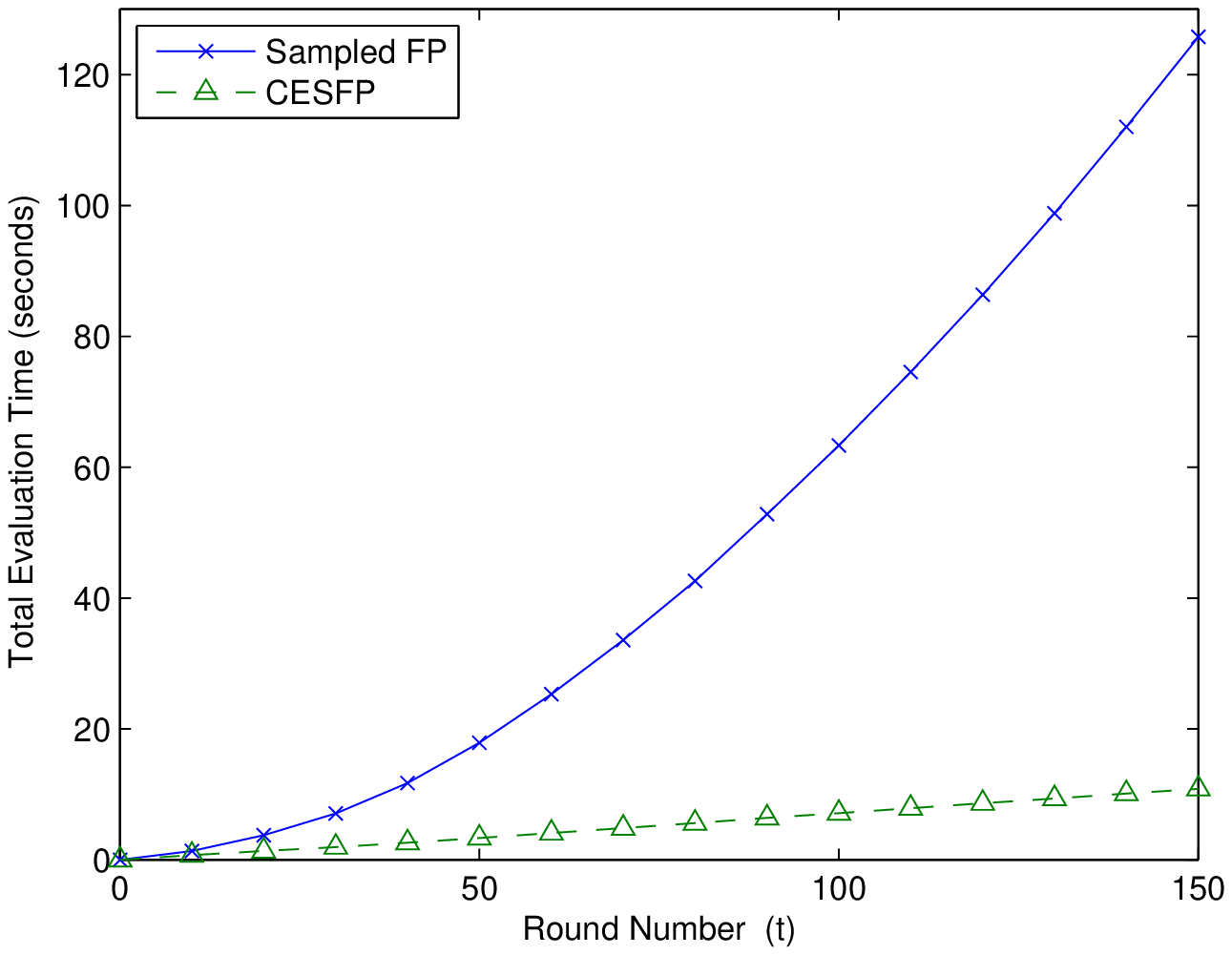}\label{fig_eval_time}}\hspace{0.0cm}
    \subfigure[]{\includegraphics[width=4.35cm]{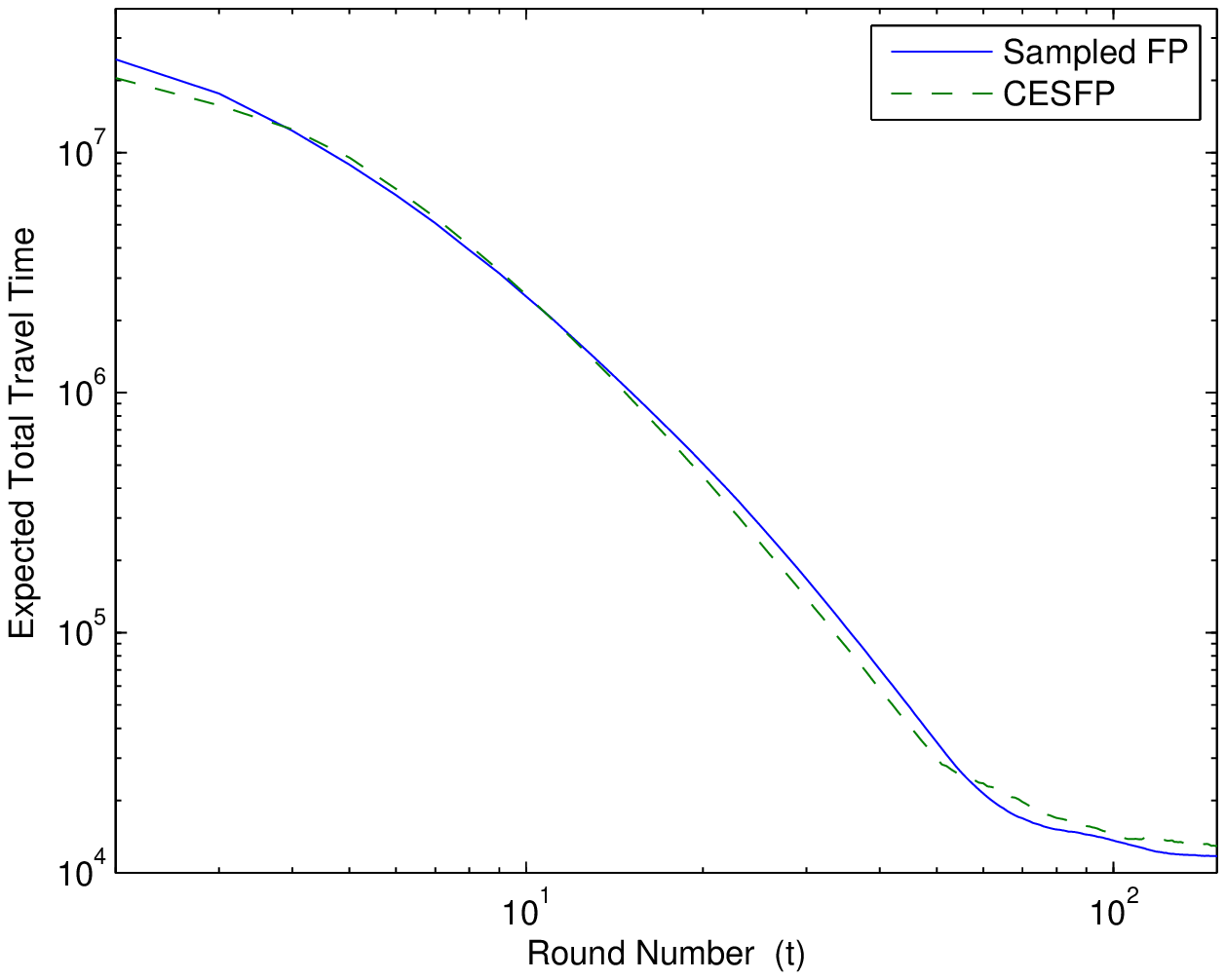}\label{fig_util}}
\caption{(a) Wall clock evaluation time for each algorithm; and (b) Expected total travel time for the mixed strategy $q(t)$.}\label{Figures2}
\end{figure}

Figure \ref{fig_util} shows a logarithmic plot of the expected total travel time if the current-iteration empirical distribution $q(t)$ were to be used as the joint mixed strategy. While a Nash equilibrium of the traffic routing game does not necessarily minimize total travel time, the trend shown in Figure \ref{fig_util} is consistent with convergence of $q(t)$ to NE, and suggests a comparable convergence rate (per repeated-play iteration) for both algorithms.


\section{Conclusions}
\label{sec_conclusions}
The classical Fictitious Play (FP) algorithm can be prohibitively difficult to implement in games with many players. Sampled FP \cite{Lambert01}---a Monte-Carlo based variant of FP---has previously been proposed as a method of mitigating computational complexity in large-scale implementations of FP.  Though Sampled FP does achieve mitigations in  complexity, it suffers from the drawback that the number of samples that must be gathered in each stage of the algorithm grows without bound.

The paper proposed Computationally Efficient Sampled FP (CESFP)---a variant of Sampled FP that requires only one sample to be drawn per stage of the algorithm. CESFP is shown to achieve Nash equilibrium learning in the same sense as FP.
A simulation example was used to demonstrate the computational properties of CESFP compared to Sampled FP.
The simulation example used a game with fairly simple structural properties. An interesting future research direction may be to study the relative empirical performance of Sampled FP and CESFP in games with more complex structure (e.g. \cite{garcia2000fictitious}).

\section*{Appendix}
{\small
\begin{lemma}
\label{lemma0}
Let $\{ x_t \}_{t \geq 0}$ satisfy $x_t \rightarrow x$ as $t\rightarrow\infty$.
Let $\{\rho_t \}_{t \geq 0}$ satisfy $0 < \rho_t \leq 1$  and $\sum_{t \geq 0} \rho_t = \infty$.
Then, the sequence $\{ y_t \}_{t \geq 0}$ given by $y_{t} = ( 1 - \rho_t ) y_{t-1} + \rho_t x_t, \quad t\geq 1,$
satisfies $y_t \rightarrow x$ as $t\rightarrow\infty$.
\end{lemma}

\begin{proof}
The result follows from Toeplitz's lemma \cite{loeve1977}.
\end{proof}

\begin{lemma}
\label{lemma1}
Let $\{\rho_t \}_{t \geq 1}$ satisfy $0 < \rho_t \leq 1$,
$\sum_{t \geq 1} \rho_t = \infty$ and $\sum_{t \geq 1} \rho_t^2 < \infty$.
Let $\{ {\mathcal F}_t \}_{t \geq 1}$ be a filtration and let $\{ X_t
\}_{t \geq 1}$ be a sequence of bounded random variables, adapted to the
filtration, say, $| X_t | \leq B$. Let $\mu_t = {\mathbb E}( X_t \,|\, {\mathcal
F}_{t-1} )$ and assume that $\left( \frac{1}{\rho_t} -1 \right) (\mu_t - \mu_{t-1})
\rightarrow 0$ almost surely. Then, the sequence of random variables $\{ \widehat
\mu_t \}_{t \geq 0}$ given by $
\widehat \mu_t = ( 1 - \rho_t ) \widehat
\mu_{t-1} + \rho_t X_t, \quad t \geq 1,$
satisfies $|\widehat
\mu_t - \mu_t| \rightarrow 0$ almost surely.
\end{lemma}

\begin{proof}
Subtracting $\mu_t$ from both sides of $\widehat \mu_t = ( 1 - \rho_t ) \widehat
\mu_{t-1} + \rho_t X_t$ gives $E_t = ( 1 - \rho_t ) E_{t-1}
+ \rho_t \left( X_t - \mu_t + \left( \frac{1}{\rho_t} - 1 \right) \delta_t \right),$ where
$E_t := \widehat \mu_t - \mu_t$ and $\delta_t := \mu_{t-1} - \mu_{t}$, $\delta_0 := 0$.

Introduce the $\mathcal{F}_{t}$-adapted sequences, for $t\geq 1$
\vskip-10pt
\begin{eqnarray*}
F_t & = & ( 1 - \rho_t ) F_{t-1} + \rho_t ( X_t - \mu_t ), \quad F_1 = E_1 \\
G_t & = & ( 1- \rho_t ) G_{t-1} + \rho_t \left( \frac{1}{\rho_t} - 1 \right) \delta_t, \quad G_1 = 0,
\end{eqnarray*}
\vskip-3pt
and note that $E_t = F_t + G_t$. We will now show that $F_t \rightarrow 0$ and $G_t \rightarrow 0$ almost surely.

By assumption, $\left( \frac{1}{\rho_t} - 1 \right) \delta_t \rightarrow 0$ almost surely. Lemma \ref{lemma0} applied to $\{ G_t \}_{t \geq 1}$ gives $G_t \rightarrow 0$ almost surely.

On the other hand,
\vskip-10pt
\begin{align*}
{\mathbb E}( F_t^2 \,|\,{\mathcal F}_{t-1}) & = ( 1 - \rho_t )^2 F_{t-1}^2 + \rho_t^2
{\mathbb E}\left( ( X_t - \mu_t)^2 \, | \, {\mathcal F}_{t-1} \right)\\
& \leq
( 1 - \rho_t )^2 F_{t-1}^2 + \rho_t^2 4 B^2  \\ & = \left( 1 + \rho_t^2 \right) F_{t-1}^2 -
2 \rho_t F_{t-1}^2 + \rho_t^2 4 B^2.
\end{align*}
\vskip-3pt
\noindent Since $\sum_{t \geq 0} \rho_t^2 < \infty$, from the Robbins-Monro Lemma \cite{robbins1985convergence} we conclude that, almost surely, $\{ F_t^2 \}_{t \geq 1}$ converges and $\sum_{t \geq 0} \rho_t F_{t-1}^2 <
\infty$: these two properties imply $F_t \rightarrow 0$ almost surely.
\end{proof}

\begin{lemma}
\label{lemma_a_rho_implication}
Let $\{\rho_t\}_{t\geq 1}$ be such that $0 < \rho_t \leq 1$ and $\lim_{t\rightarrow\infty} \frac{1}{t\rho_t} = 0$. Then $\sum_{t\geq 1} \rho_t = \infty$.
\end{lemma}
\begin{proof}
We claim there exists $c,T > 0$ such that $\rho_t \geq c\frac{1}{t}$ for all $t\geq T$. If this were not so, then for every $c>0$ there would hold $\rho_t < c\frac{1}{t}$ infinitely often, which would imply $\frac{1}{t\rho_t} > \frac{1}{c}$ infinitely often---contradicting the hypothesis that $\lim_{t\rightarrow\infty} \frac{1}{t\rho_t} = 0.$

Thus, there exists $c,T > 0$ such that $\rho_t \geq c\frac{1}{t} ~\forall t\geq T$, and hence $\sum_{t\geq 1} \rho_t \geq \sum_{t\geq T} \rho_t \geq c\sum_{t\geq T} \frac{1}{t} = \infty.$
\end{proof}
}

\bibliographystyle{IEEEtran}
\bibliography{myRefs}

\end{document}